\documentclass[12pt]{article}
\usepackage[english]{babel}    
\usepackage[latin1]{inputenc}
\usepackage{amsmath,amssymb,,amsthm,latexsym}
\usepackage{enumerate}
\usepackage{mathrsfs}
\usepackage{xfrac}
\newcommand{\algA}{\mathcal{A}}
\newcommand{\A}{\mathcal{A}}

\newcommand{\F}{F}

\newtheorem{teo}{Theorem}
\newtheorem{Le}{Lemma}

\newtheorem{Ex}{Example}
\newtheorem{Que}{Question}
\newtheorem{Con}{Conjecture}

\DeclareMathOperator{\Ele}{Ele}

\DeclareMathOperator{\rank}{rank}

\DeclareMathOperator{\pf}{pf}
\DeclareMathOperator{\diag}{diag}
\DeclareMathOperator{\Id}{Id}
\DeclareMathOperator{\adj}{adj}
\DeclareMathOperator{\Aut}{Aut}

\DeclareMathOperator{\Skew}{Skew}
\title{Antisymetric matrices and Lotka-Volterra algebras}
\author{Juan C. Gutierrez Fernandez\\
Instituto de Matem\'atica e Estat\'{\i}stica, Universidade de S\~ao Paulo,\\ Rua do Mat\~ao, 1010, CEP 05508-090, S\~ao Paulo, Brazil \\
e-mail: jcgf@ime.usp.br
\\ and \\
Claudia I. Garcia\\
EACH, Universidade de São Paulo, CEP 03828-000, São Paulo, Brazil\\ e-mail: claudiag@usp.br}
\date{}

\begin{document}
\maketitle

\begin{abstract}
The purpose of this paper is to study the structure of  Lotka-Volterra algebras, the set of their idempotent elements  and their group of automorphisms. These algebras are defined through  antisymmetric matrices and  they emerge in connection with biological  problems and  Lotka-Volterra systems for the interactions of neighboring individuals.

\end{abstract}




\section{Introduction and Preliminaries} 
Let $\algA$ be a commutative algebra (not necessary associative) over a field $F$ of characteristic different from 2.  The algebra $\A$ is \it baric \rm if it admits a nontrivial homomorphism $\omega \colon \A \longrightarrow \F$. The homomorphism $\omega$ is called \it weight homomorphism\rm. The \it hiperplane unit \rm and the \it barideal \rm of $\A$ are defined by $H=\{ x\in \A \,\, | \, \,\omega(x)=1\}$ and $
N= \ker \omega = \{ x\in \A \,\, | \,\, \omega(x)=0\}$,
respectively.
The algebra $\A$ is called a Lotka-Volterra  algebra if it admits a basis $\Phi=\{e_1,\ldots,e_n\}$ and an antisymmetric matrix $A=(a_{ij})$ with $n$ rows and $n$ columns where the  entries $a_{ij}$ are in the field $F$ such that 
$e_ie_j = (\sfrac{1}{2} + a_{ij}) e_i + (\sfrac{1}{2} + a_{ji}) e_j$
for all $i,j\in \{1,2,\ldots,n\}$.   Such a commutative algebra is called \it Lotka-Volterra algebra \rm associated to the antisymmetric matrix $A$ with respect to the basis $\Phi=\{e_1,\ldots,e_n\}$. The basis $\Phi$ is called \it natural \rm basis of $\A$. Some results about its structure can be found in \cite{Gani,Itoh,Yoon0,Yoon,Micali0,Micali,Micali1}. We can easily see that  every Lotka-Volterra algebra is baric with weight homomorphism defined by  $$\omega\Big(\sum_i\lambda_i e_i\Big) = \sum_i \lambda_i.$$ 
For example,  Lokta-Volterra algebras associated to the zero matrix are also called  normal Bernstein algebras. These algebras are essential in solving the classical Bernstein's problem, as well as in the genetic and biological interpretation of Bernstein operators (see \cite{Lybook,Holgate,Juan,Juan1}). On the other hand, these algebras emerge in connection with a Lotka-Volterra system for the interactions of neighboring individuals. Following Kimura \cite{Kimura} and Mather \cite{Mather}, the system for binary interaction is represented by 
$\frac{d}{d t } p_i(t) = p_i(t)\sum_{j=1}^n a_{ij} p_j(t)$, for $t\geq t_0$, where $a_{ij}+a_{ji}=1$,
$p_i(t_0)>0$, $ \sum_{j=1}^n p_j(t_0)=1$ for   $i=1,2,\ldots,n.$ Previous quadratic differential equations  can be  represented by 
$\frac{d}{d t} p(t)= \frac{1}{2}\Big(p(t)\circ p(t) - p(t)\Big)$
where $(\A,\circ)$ is the $n$-dimensional Lotka-Volterra algebra associated to the antisymmetric matrix $(a_{ij})$. In book \cite{Walcher}, we can find some  connections between quadratic differential systems and the corresponding nonassociative algebra.

From a  mathematical  and biological perspective, a very interesting question  about Lotka-Volterra algebras is whether or not their weight functions are uniquely determined. It is well know that when $N$ is nil, then the weight function is uniquely determined (see \cite{Reed, Busekros}). In order to exhibit at least a sufficient condition for a Lotka-Volterra algebra $\A$ to have a unique weight function, we must first discuss the issue of idempotent elements. 
In \cite{Yoon0} and \cite{Micali}  the authors start by studying the set $\text{Id}(\A)$, of  nonzero idempotent elements of $\A$. Firstly, we see  that every idempotent element $e$ of $\A$ has weight either $0$ or $1$ since $\omega(e) = \omega(e^2) = \omega(e)^2$. Assume now that the Lotka-Volterra algebra $\A$ has another  nonzero weight homomorphism $\chi$. Because every idempotent element of $\A$ has weight either 0 or 1, the homomorphism $\chi$ determines a partition of  $\Phi$ into two nonempty subsets  $I= \Phi \cap \ker \chi$ and $J= \{ e_j \in \Phi \, | \chi(e_j)=1\}$. For every $e_i\in I$ and $e_j\in J$ we have that $\chi(e_ie_j) = \chi(e_i)\chi(e_j) = 0\cdot 1= 0$ and hence 
$e_ie_j = e_i$. This proves  first part of the following lemma (see \cite{Gani}).
\begin{Le} \label{l1} Let $\A$ be a Lotka-Volterra algebra associated to an antisymmetric matrix $(a_{ij})$ with respect to the basis $\Phi=\{e_1,\ldots, e_n\}$. The weight homomorphism $\omega$ of $\A$ is uniquely determined,   if and only if the basis $\Phi$ does not  have a  nontrivial partition $\Phi = I \cup J$ where $e_ie_j = e_i$ for all $i\in I$ and  $j\in J$.
\end{Le}

The paper is organized as follows. After this introductory section we first
recall some known results about idempotent elements of a Lotka-Volterra algebra and next we describe the set  of all idempotent element of Lotka-Volterra algebras with dimension 3  and 4. Using properties of antisymmetric matrices and their pfaffians we get new  results about the idempotent elements of generic Lotka-Volterra algebras. In Section 3, we give new properties about  the group of automorphisms of a Lotka-Volterra algebra and finally we  describe in Section 4 the set of all the automorphisms of a Lotka-Volterra algebra with dimension 3.

\section{Idempotent elements}
For the structure of an algebra, the existence of an idempotent element provides a direct sum decomposition of the algebra. Furthermore, from a biological aspect, idempotent elements  also may  have genetic or hereditary significance (see  \cite{Reed} and \cite{Busekros}).
In the following, we will denote by $\text{Id}_k(\A)$ the set $\{ x\in \text{Id}(\A) \,\,|\,\, \omega(x)=k\}$ for $k=0,1$.  Obviously,  $\text{Id}(\A)$ is the disjoint union of the subsets   $\text{Id}_0(\A)$ and $\text{Id}_1(\A)$. By  Proposition 1.1 of \cite{Yoon}, for $x = \sum_{i=1}^n \lambda_i e_i$ and $y= \sum_{i=1}^n \mu_i e_i$ where  $\lambda_i,\mu_i\in F$,  the product $xy$ can be written as follows 
$$xy = \frac{1}{2}\Big(\omega(y)x + \omega(x)y\Big) + \sum_{i=1}^n (\lambda_i \omega_i(y) + \mu_i\omega_i(x)) e_i,$$
where $\omega_i: \A \longrightarrow F$ is a linear mapping defined by $\omega_i(e_j) = a_{ij} $ for $j=1,2,\ldots,n$. In particular, for  $x=y$ we have the following relation
\begin{equation}\label{x2}
x^2 = \omega(x) x + 2 \sum_{i=1}^n \lambda_i\omega_i(x) e_i.
\end{equation}
The \it  support \rm  of   $x$  with respect to the basis $\{ e_1,\ldots, e_n\}$, denoted by $\text{supp}(x)$, is the set of indices for which the corresponding coefficient is nonzero, that is, $e_i\in \text{supp}(x)$ if and only if $\lambda_i\ne 0$. Using relation (\ref{x2}),  we can reduce the problem of determining the idempotent elements   of $\A$ to solving  the following two linear systems:
\begin{equation}\label{id0} x\in \text{Id}_0(\A) \Longleftrightarrow
 \omega(x) =0\text{ and }
 \omega_i(x)=\frac{1}{2} \, \text{  for all } i\in  \text{supp}(x).
 \end{equation}
 and 
\begin{equation} \label{id1} x\in \text{Id}_1(\A) \Longleftrightarrow
 \omega(x) =1\text{ and }
 \omega_i(x)=0 \, \text{  for all } i\in  \text{supp}(x).
  \end{equation}
 Now we will analyze particular cases and Lotka-Voltera algebras with small dimension. If $A$ is the null matrix, then $\A$ is also called normal Bernstein algebra and every element of weight 1 is an idempotent element; furthermore, $\A$ has  no nonzero idempotent elements with   zero weight. Assume now that $a_{ij} = a$ with $a\ne 0$ for all $i,j$ with $1\leq i<j \leq n$. Following \cite{Micali}, for every nonempty set $I=\{p_1,\ldots,p_k\}$  where  $1\leq p_1<p_2<\cdots <p_k \leq n$, the algebra $\A$ has exactly one idempotent element with support equal to the set $I$. This idempotent element is equal to 
  $\frac{1}{2a}\sum_{i=1}^k (-1)^i e_{p_i}$ if $k$ is even and it is equal to  $\sum_{i=1}^k (-1)^{i+1} e_{p_i}$ if $k$ is odd. In particular, $\A$ has exactly $2^n$ idempotent elements.

Let now $\A$ be a Lotka-Volterra algebra associated to an $n\times n$ antisymmetric matrix $A=(a_{ij})$. Let $I=\{p_1,\ldots,p_k\}$  be a subset of $\{1,2,\ldots, n\}$ with $k$ elements where  $1\leq p_1<p_2<\cdots <p_k \leq n$. Denote by $A_I$, the submatrix  $k\times k$ of  $(a_{ij})$ given by $(a_{p_ip_j})$. If the determinant of $A_I$ is different from zero, then we can affirm, from (\ref{id0}) and (\ref{id1}), that $\A$ has at most one idempotent element with weight zero and support $I$, and also $\A$ does not   have  idempotent elements with weight 1 and support $I$. Take now $i$ and $j$, two  positive integers where $1\leq i<j\leq n$. By \cite{Yoon0} and \cite{Micali}, 
 for $a_{ij}=0$ we have that  every element of the form $\alpha e_i + \beta e_j$ with $\alpha +\beta = 1$ is an idempotent element of $\A$, and $\A$ does have no nonzero idempotent elements with support $\{i,j\}$ and weight 0. In the case $a_{ij}\ne 0$, $\A$ does have no idempotent elements with weight 1 and exactly one idempotent element with  weight 0  and support $\{i,j\}$, it is given  by
$$\frac{1}{2a_{ij}}(-e_i+e_j).$$

Using relations (\ref{id0}) and (\ref{id1}) we can easily prove the following. 
\begin{Le} \label{l2} Let  $\A$ be a Lotka-Volterra algebra associated to an $n\times n$ antisymmetric matrix $A=(a_{ij})$ aver a field with at least 4 elements (and characteristic $\ne 2$). Take $I=\{i,j,k\}$ where $1\leq i<j<k\leq n$ and  $a_{ij}\ne 0$. Then $A_I$ is of the form
$$A_I= \left[\begin{array}{ccc}0&a&b\\-a&0&c\\-b&-c&0\end{array}\right]
$$ with $a\ne 0$. We have the following properties.
\begin{enumerate}[(i)]
	\item $\A$ has  idempotent elements with weight 1  and support $I$ if and only if $b\ne a+c$ and $abc\ne 0$. In this case, $\A$ has exactly one  idempotent element with support $I$ and it  is equal to 
	$$\frac{1}{a+c-b}\big(c e_i-be_j+ae_k\big).$$
\item $\A$ has  idempotent elements with weight 0  and support $I$ if and only if $b=a+c$. This set of idempotent elements is equal to  ($F^*=F\setminus \{0\}$)
$$ \Big\{ \frac{2c\lambda -1}{2a}e_i +\frac{1-2b\lambda}{2a}e_j + \lambda e_k \,\, | \,\, \,2c\lambda \ne 1, \, 2b\lambda\ne 1  \text{ and } \lambda \in F^*\Big\}.$$ 
\end{enumerate}
\end{Le}

We will now describe idempotent elements  with exactly  4 elements in your support. Let $I=\{ i,j,k,t\}$ where $1\leq i<j<k<t\leq n$ and 
$$A_I = \left[\begin{array}{cccc}0&a_1&b_1&c_1\\-a_1&0&a_2&b_2\\
-b_1&-a_2&0&a_3\\-c_1&-b_2&-a_3&0\end{array}\right]\ne 0.$$
By rearranging   the basis we can assume that $a_1\ne 0$. Take $e= \lambda_1 e_i + \lambda_2 e_j + \lambda_3 e_k +\lambda_4 e_t\in \A$ with $\lambda_i\in F^*$. Using   relations (\ref{id0}) and (\ref{id1}) we can easily prove   that $e\in \text{Id}(\A)$ if and only if $X=(\lambda_1, \lambda_2,\lambda_3,\lambda_4)$ is a solution of one of the following two linear systems:
$$M X^t = (0,\sfrac{1}{2},0,\delta_4,\delta_3)^t ,\qquad\text{and}\qquad  
M X^t = (1,0,a_1,0,0)^t,$$
where 
$$M= \left[\begin{array}{cccc}
1&1&1&1\\
0&a_1&b_1&c_1\\
0&0&\delta_4&\delta_3\\
0&0&0&2\Delta\\
0&0&-2\Delta&0\end{array}\right],$$
 $\delta_4=a_1+a_2-b_1$,\, $\delta_3=a_1+b_2-c_1$ and 
$\Delta= a_1a_3+a_2c_1-b_1b_2$ is the pfaffian of $A_I$. Solving above linear systems, we get the following results.
\begin{enumerate}
\item If $\Delta \ne 0$, then $\A$ has at most one idempotent element with support $I$.  This idempotent element exists if and only if $\delta_l\ne 0$ for $l=1,2,3,4$ where $\delta_1= a_2+a_3-b_2$ and $\delta_2=a_3+b_1-c_1$. In this case, the idempotent element has weight 0 and coordinates
$$\lambda_l= (-1)^l\frac{\delta_l}{2\Delta}, \quad (l=1,2,3,4).
$$

\item If $\Delta = 0$ and  $\delta_4\ne 0$, then every idempotent element with support $I$  has weight 1, and the coordinates are as follows
$$\lambda_3= \frac{a_1-\delta_3\lambda_4}{\delta_4},\quad
\lambda_2 = -\frac{1}{a_1}\Big(b_1\lambda_3+c_1\lambda_4\Big),\quad
\lambda_1=1-\lambda_2-\lambda_3-\lambda_4,\, \text{ for some }\lambda_4\in F^*.
$$
\item If $\Delta = \delta_4=0$ and  $\delta_3\ne 0$, then every idempotent element with support $I$  has weight 1, and the coordinates are as follows
$$\lambda_4= \frac{a_1}{\delta_3},\quad
\lambda_2 = -\frac{1}{a_1}\Big(b_1\lambda_3+c_1\lambda_4\Big),\quad
\lambda_1=1-\lambda_2-\lambda_3-\lambda_4,\, \text{ for some }\lambda_3\in F^*.$$
\item If $\Delta = \delta_4=\delta_3=0$, then every idempotent element with support $I$  has weight 0, and the coordinates are as follows
$$
\lambda_2 = \frac{1}{a_1}\Big(\frac{1}{2}-b_1\lambda_3-c_1\lambda_4\Big),\quad
\lambda_1=-\lambda_2-\lambda_3-\lambda_4,\, \text{ for some }\lambda_3,\lambda_4\in F^*.$$
\end{enumerate}
\begin{Que} \rm If $\A$ has one idempotent element with support $I$ and weight 1, then every idempotent element of $\A$ with support $I$ has weight 1?
\end{Que}

Here, we confirm this question in a positive way, when  $I$ has cardinality at most 4.
 
Let  $A$ be an $n\times n$ antisymmetric matrix. Then $A_{ij}$ denotes  the square submatrix of $A$ formed by deleting the $i$-th row and $j$-th column, $A_i$ denotes  the submatrix  $A_{ii}$,  $A_{\hat{\imath},\hat{\jmath}}$ is the submatrix  of $A$ obtained when  $i$-th and $j$-th  rows and columns are removed, and  pf$(A)$ is the pfaffian of $A$. We have that $\det(A)= \pf(A)^2$ and the pfaffian can  be computed
recursively as $ \pf(A) = \sum_{j=1}^n(-1)^{i+j+\theta(j-i)} a_{ij} \pf(A_{\hat{\imath}\hat{\jmath}})$ where index $i$ can be selected arbitrarily. Provably, the following two lemmas are well known. Here, we will give an elementary proof. 

\begin{Le}\label{l3} If $n$ is even, $n>2$ and $i\ne j$, then  $\det(A_{ij})=(-1)^{\theta(j-i)}\pf(A)\pf(A_{\hat{\imath},\hat{\jmath}})$, where $\theta$ is the Heaviside step function.
\end{Le}
\begin{proof} Let $t_{ij}, \, (1 \leq i< j \leq n)$,  be $n(n-1)/2$  algebraically independent elements over $F$, let $t_{ii}=0$ and $t_{ij}=-t_{ji}$ for $i>j$. Then the matrix $T=(t_{ij})$ is antisymmetric. Define a new antisymmetric matrix $S=(s_{ij})$ by 
	$$s_{ij}= (-1)^{i+j+\theta(i-j)}  \pf(T_{\hat{\imath},\hat{\jmath}})$$
	for $i\ne j$ and $s_{ii}=0$. The $(i,j)$-entry of the product $TS$ is equal to ($\delta_{ij}$ is  Kronecker delta)
	$$\sum_{k=1}^n t_{ik}s_{kj} = \sum_{ \substack{k=1 \\ k \neq j}}^n  (-1)^{k+j+\theta(k-j)}t_{ik} \, \pf(T_{\hat{\jmath},\hat{k}})= \delta_{ij} \pf(T)
	$$ 
since for $i=j$ we have  the recursive formula for  the pfaffian of $T$, and for $i\ne j$ the above  sum  is equal  to the pfaffian of a matrix $Q$ obtained from $T$ by  first replacing the $j$-th row by the $i$-th row, next  replacing the $j$-th column by the $i$-th column; since $Q$ has two equal rows,  $\pf(Q)=0$.  Thus, 
	$\frac{1}{\pf(T)}S$ is the inverse matrix of $T$. Because  $T^{-1}=\frac{1}{\det(T)} \adj (T)$,  we have the following relations
	$$\frac{1}{\pf(T)^2}(-1)^{i+j} \det(T_{ij})= \frac{1}{\pf(T)}s_{ji} = \frac{1}{\pf(T)}
	(-1)^{i+j+\theta(j-i)}  \pf(T_{\hat{\imath},\hat{\jmath}})$$
	 for all $i,j\in \{1,2,\ldots, n\}$. Consequently, 
	$$\det(T_{ij})= (-1)^{\theta(j-i)} \pf(T) \cdot \pf(T_{\hat{\imath},\hat{\jmath}}).$$
	This completes the proof of the lemma. 
\end{proof}

\begin{Le} If $n$ is odd, then $\det(A_{ij})=\pf(A_i)\pf(A_j)$ for all $i,j\in \{1,2,\ldots,n\}$.
\end{Le}
\begin{proof} The result is obvious when $i=j$ so assume $i\ne j$. We can assume without loss of generality that  $i=1$ and $j=2$ since  for an arbitrary antisymmetric matrix, simultaneous interchange of two different rows and corresponding columns changes the sign of the pfaffian. Thus, the matrix $A$ can be written as  
$$A=\left(\begin{array}{cc|c}
0 & a_{12}& u\\
-a_{12} & 0 & v \\ \hline
-u^t & - v^t & B
\end{array}\right)$$
where $B$ is an $(n-2)\times (n-2)$ antisymmetric matrix,  $u=(a_{13}, \cdots, a_{1n})$ and  $v=(a_{23}, \cdots, a_{2n})$. There exists a non-singular matrix $Q$ such that $Q^tBQ =N=\diag\{S_1,\cdots,S_{m},{\bf 0}\}$ where $m=(n-3)/2$, $S_k= \left(\begin{array}{cc}
0 & a_k\\
-a_k & 0 
\end{array}\right)\in M_2(F)$ and ${\bf 0}=(0)\in M_1(F)$. Let $R=\diag\{1,Q\}$ and $d=\det R= \det Q\ne 0$. Then 
\begin{eqnarray*}
\det(A_{12}) &=& \frac{1}{d^2} \det(R^t A_{12}R )=  \frac{1}{d^2}  \det \Big( R^t\left(\begin{array}{c|c}
-a_{12}  & v \\ \hline
-u^t  & B
\end{array}\right)R \Big)\\
&=&
\frac{1}{d^2} \det \left(\begin{array}{c|c}
-a_{12}  & vQ \\ \hline
-Q^tu^t  & Q^tBQ
\end{array}\right)= \frac{1}{d^2} bc \Pi_{k=1}^m a_k^2
\end{eqnarray*}
 where 
$b= \sum_{k=3}^n a_{1k}q_{k-2,n-2}$ and $c= \sum_{k=3}^n a_{2k}q_{k-2,n-2}$. On the other hand, 
\begin{eqnarray*}
\pf(A_1) &=& \frac{1}{d} \pf(R^tA_1R)=
 \frac{1}{d} \pf\Big(R^t \left(\begin{array}{c|c}
0  & v \\ \hline
-v^t  & B
\end{array}\right)R\Big)\\
&=&\frac{1}{d} \pf \left(\begin{array}{c|c}
0  & vQ \\ \hline
-Q^tv^t  & Q^tBQ
\end{array}\right)=\frac{1}{d} c (\Pi_{k=1}^m a_k)
\end{eqnarray*}
by recursive formula of the Phaffian. Analogously we can see that
$$\pf(A_2) = pf \left(\begin{array}{c|c}
0  & u \\ \hline
-u^t  & B
\end{array}\right)=\frac{1}{d} \pf \left(\begin{array}{c|c}
0  & uQ \\ \hline
-Q^tu^t  & Q^tBQ
\end{array}\right)=\frac{1}{d} b (\Pi_{k=1}^m a_k).$$
 This completes the proof of the lemma.
\end{proof}

\begin{teo}\label{t1} Let $\A$ be an $n$-dimensional Lotka-Volterra algebra associated to an antisymmetric matrix $A=(a_{ij})$ with respect to the basis $\Phi=\{e_1,\ldots, e_n\}$. Take $I=\{1,2,\ldots,n\}$.
\begin{enumerate}[(i)]
\item If $n$ is even and $\det(A)\ne 0$, then $\A$ has at most  one idempotent element with support $I$; in this case, the idempotent element has weight  0, and its coordinates are as follows
\begin{equation}\label{t12-0}
	\displaystyle \lambda_j= \frac{1}{2\pf(A)}\sum_{\substack{i=1 \\ i \neq j}}^n (-1)^{i+j+\theta(j-i)} \pf(A_{\hat{\imath},\hat{\jmath}}), \qquad (j=1,2,\ldots, n).
\end{equation}
\item If $n$ is odd, $\sum_{i=1}^n (-1)^i \pf(A_{i}) \ne 0$  and $ \pf(A_{j})\ne 0$ for $j=1,2,\ldots, n$, then  $\A$ has exactly  one idempotent element with support $I$; this idempotent element  has weight  1, and its coordinates are as follow
\begin{equation}\label{t12}
\lambda_j= \frac{(-1)^j \pf(A_{j})}{\sum_{i=1}^n (-1)^i \pf(A_{i})}, \qquad (j=1,2,\ldots, n).
\end{equation}
\end{enumerate}
\end{teo}

\begin{proof}  If  there exists  $e=\sum_{i=1}^n \lambda_i e_i\in \Id(\A)$ with $ \lambda_i\in F^*$ for $i=1,2\ldots,n$, then relations  (\ref{id0}) and (\ref{id1}) imply  that $(\lambda_1,\ldots,\lambda_n)$ is a solution of one of the following two  matrix equations:

\begin{equation}\label{idm}
\left(\begin{array}{ccc}
1&\cdots&1   \\ \hline
 &      & \\
 &   A  & \\
  &      & 
\end{array}\right) 
\left(\begin{array}{c}
x_1   \\ 
 \vdots\\
x_n
\end{array}\right) = 
\left(\begin{array}{c}
0   \\  \hline
1/2\\
 \vdots\\
1/2
\end{array}\right)
\text{  or }
\left(\begin{array}{ccc}
1&\cdots&1   \\ \hline
 &      & \\
 &   A  & \\
  &      & 
\end{array}\right) 
\left(\begin{array}{c}
x_1   \\ 
 \vdots\\
x_n
\end{array}\right) = 
\left(\begin{array}{c}
1   \\  \hline
0\\
 \vdots\\
0
\end{array}\right).
\end{equation}

Firstly, we will assume that  $\det(A)\ne 0$. Then $n$ is even,  second system of (\ref{idm})  has no solutions and using Cramer's rule,  first matrix equation of (\ref{idm})  has a unique solution  given by 
\begin{eqnarray*}x_j &=& \frac{1}{\det(A)} \sum_{i=1}^n (-1)^{i+j} \frac{1}{2}\det(A_{ij})= \frac{1}{2\pf(A)^2} \sum_{\substack{i=1 \\ i \neq j}}^n (-1)^{i+j} \det(A_{ij})\\
&=&\frac{1}{2\pf(A)^2} \sum_{\substack{i=1 \\ i \neq j}}^n (-1)^{i+j}(-1)^{\theta(j-i)}\pf(A) \pf(A_{\hat{\imath},\hat{\jmath}}) \quad \text{by Lemma \ref{l3}}\\
&=&\frac{1}{2\pf(A)} \sum_{\substack{i=1 \\ i \neq j}}^n (-1)^{i+j+\theta(j-i)} \pf(A_{\hat{\imath},\hat{\jmath}}) \quad \text{for } j=1,2,\ldots,n.
\end{eqnarray*}

Assume now that  $n$ is odd, $d := \sum_{i=1}^n (-1)^i \pf(A_{i}) \ne 0$  and $ \pf(A_{j})\ne 0$ for $j=1,2,\ldots, n$. 
Becuase $n$ is odd, we have that  $\det(A)=0$. Take  $B= \tiny \left(\begin{array}{c}
\bf{1}   \\ \hline
 A
\end{array}\right)  \in M_{n+1,n}(F)$,   the coefficient matrix of every matrix system given in  (\ref{idm}), and let $B_{(i)}$ be the submatrix of $B$  obtained by deleting $(i+1)$-th row. Using Laplace expansion, we get  that 
$$\det(B_{(i)})= \sum_{j=1}^n (-1)^{i+1}(-1)^{i+j} \det(A_{ij})=
\sum_{j=1}^n (-1)^{j+1} \pf(A_{i})\pf(A_{j})= -
\pf(A_{i})\cdot d$$
for $i=1,2,\ldots,n$. Thus, matrix $B$ has rank  equal to $n$ and hence  each matrix equation of (\ref{idm}) has at most one solution. Furthermore, if $\lambda_j= (-1)^j\pf(A_j)/d$, then obviously $\sum_{j=1}^n \lambda_j=1$ and 
\begin{eqnarray*}
(-1)^i \pf(A_i)\sum_{j=1}^n a_{ij} \lambda_j&=& 
\frac{1}{d}\sum_{j=1}^n (-1)^{i+j}a_{ij}\pf(A_i)\pf(A_j)\\
& =& \frac{1}{d}\sum_{j=1}^n (-1)^{i+j}a_{ij}\det(A_{ij})= \frac{1}{d} \det(A)= \frac{1}{d} \cdot 0=0
\end{eqnarray*}
 so that $(\lambda_1, \ldots, \lambda_n)$
 is the unique solution of the second matrix equation of (\ref{idm}). On the other hand, because $\det(A)=0$ and $\rank(B)=n$, we can affirm that the vector $(1,\ldots,1)$ is not a linear combination of the rows of $A$. Using now  that $A$ is antisymmetric, we can affirm  that  $(1,\ldots,1)$, and hence also $(\sfrac{1}{2},\ldots,\sfrac{1}{2})$, is not a linear combination  of the columns of $A$. Therefore, first matrix equation of (\ref{idm}) has no solutions. This completes the proof of the theorem.
\end{proof}
Note that (ii) of above theorem can be generalized in the following way.

\begin{Le} Let $\A$ be as in previous theorem. If $n$ is odd and $d:=\sum_{i=1}^n (-1)^i \pf(A_{i}) \ne 0$, then  $\A$ has at most  one idempotent element with support $I$; in this case, this idempotent element  has weight  1, and its coordinates are given by (\ref{t12}).
\end{Le}
\begin{proof} Let $t_{ij} \, (1 \leq i< j \leq n)$,  be $n(n-1)/2$  algebraically independent elements over $F$, let $t_{ii}=0$ and $t_{ij}=-t_{ji}$ for $i>j$. Then  $T=(t_{ij})$ is an antisymmetric matrix. Obviously 
	$d(T):=\sum_{i=1}^n (-1)^i \pf(T_{i}) \ne 0$, and $ \pf(T_{i})$ is a nonzero polynomial of degree $(n-1)/2$ in $t_{ij} \, (1 \leq i< j \leq n)$. As in previous theorem, we have that 
	$$\pf(T_i)\Big(\sum_{j=1}^n (-1)^j t_{ij}\pf(T_j)\Big)=0$$
for $i=1,2,\ldots, n$. Because, $\pf(T_i)	\ne 0$, we have that 
$$\sum_{j=1}^n (-1)^j t_{ij}\pf(T_j)=0$$
	for $i=1,2,\ldots, n$. This completes the proof of the lemma.
\end{proof}

\begin{Ex} \rm Let $\A$ be the Lotka-Volterra algebra of dimension 3 over the real field, with natural  basis $\Phi=\{e_1,e_2,e_3\}$ and associated to the skew-symmetric matrix $A=(a_{ij})$ where $a_{12}=-a_{21}=1$ and $a_{ij}=0$ in other cases. This algebra has odd dimension, $\sum_{j=1}^3 (-1)^j \pf(A_j) = (-1)^3\pf(A_3) = -1 \ne 0$ and it does have no idempotent elements with support $\{1,2,3\}$. 
\end{Ex}

Let $A=(a_{ij})$ be an antisymmetric matrix of order $n\times n$ such that  $a_{ij}=a$ for all $i<j$. If $n$ is even, then we can check inductively that $\pf(\A) = a^{n/2}$ since the case $n=2$ is trivial and for $n>2$, using the recursive formula of the pfaffian, we have that 
$\pf(A) = \sum_{j=2}^n (-1)^j a_{1j} \pf(A_{\hat{1}\hat{j}})=  a \sum_{j=2}^n (-1)^j  \pf(A_{\hat{1}\hat{j}})= a \pf(A_{\hat{1}\hat{2}}) = a a^{(n-2)/2}= a^{n/2}$. In particular, we have that $\det(A)=a^n$. On the other hand, if $n$ is odd then 
$$\sum_{i=1}^n (-1)^i \pf(A_{i})= \sum_{i=1}^n (-1)^ia^{(n-1)/2}= -a^{(n-1)/2}.$$
The following lemma, proved in \cite{Micali}, is now a corollary of Theorem \ref{t1}.
\begin{Le} \label{le6}
	Let $\A$ be as in Theorem \ref{t1} and assume  $a_{ij} = a\ne 0$ for all $i<j$. Then for each $J=\{p_1,\ldots,p_k\}$ with $1\leq p_1 < \cdots < p_k \leq n$ we have exactly one idempotent element e  in $\A$ with support $J$. If $k$ is  even, then $e=\frac{1}{2a}\sum_{i=1}^k (-1)^i e_{p_i}\in N$ and if $k$ is odd then  $e=\sum_{i=1}^k (-1)^{i+1} e_{p_i}\in H$. In particular, $\A$ has exactly $2^n$ idempotent elements. 
\end{Le}

\section{Automorphisms}
In the following we will denote by $\Aut (\A)$ the automorphism group of the algebra $\A$. Firstly, we observe that for each automorphism $\sigma$ of  $\A$, we have that  $\sigma(\Id(\A))= \Id(\A)$. Because the algebra $\A$ is spanned by  $\Id(\A)$, there exists a group monomorphism from $\Aut(\A)$ to $S_{\Id(\A)}$ where  $S_{\Id(\A)}$ is the set of all permutations of the set $\Id(\A)$. In particular, if  $\Id(\A)$ is finite, then $\Aut(\A)$ is finite. 
Thus, using  Lemma \ref{l1} and  Theorem \ref{t1}, we see that  in a generic Lotka-Volterra algebra of dimension $n$, there exist at most  $ {2^n-1 \choose n}$ automorphisms. Particular cases were analyzed in \cite{Yoon} and \cite{Micali}:
\begin{enumerate}
	\item If $A=(0)$, then $\Aut(\A)$ is the set of all  linear automorphisms $\sigma$ of $\A$ such that $\sigma(H)\subseteq H$.
	\item Let  $n=2$. If $a_{12}\not\in \{0,\,\sfrac{1}{2}, - \sfrac{1}{2}\}$, then $\Aut(\A) = \{ \text{Id} \}$ where $\text{Id}$ is the identity mapping. If $a_{ij}=\sfrac{1}{2}$ with $\{i,j\}=\{1,2\}$, then $\Aut(\A)= \{ \text{Id}, \sigma \}$, where $\sigma(e_i)= e_j-e_i$ and $\sigma(e_j)=e_j$.
	\item Let $I$ be a subset of $\Psi=\{1,2,\ldots,n\}$ with at least 2 elements such that $a_{ij}=0$ and $a_{ik} = a_{jk}$   for all $i,j\in I$ and all $k \in \Psi\setminus I$. Let $U$ be the vector subspace of $\A$ spanned by $\{e_i \, | \, i\in I\}$. Then the set $\Aut(\A | I)$   of all   linear mappings $f$ from  $\A$ to $\A$  such that $\{f( e_i) \, | \, i\in I\}$   is linearly independent and contained in $H \cap U $ and $f(e_k)= e_k$  for all   $k \in \Psi\setminus I$, is an infinite subgroup  of  $\Aut(\A)$.
\end{enumerate}

\begin{Que} \rm If $\Aut(\A)$ is a infinite group, then there  exists $I \subset \{1,2,\ldots,n\}$ with at least 2 elements such that $a_{ij}=0$ and $a_{ik} = a_{jk}$   for all $i,j\in I$ and $k \in  \{1,2,\ldots,n\}\setminus I$?
\end{Que}

\begin{Le}\label{laut1} If $\omega$ is uniquely determined, then $\sigma(H)=H$ for all $\sigma\in \Aut(\A)$.
\end{Le}
\begin{proof} If $\sigma\in \Aut(\A)$, then $\omega \circ \sigma$ is a nonzero homomorphism from $\A$ to $F$. Since $\omega$ is unique, we have that $\omega=\omega\circ\sigma$.
\end{proof}

\begin{Le}\label{laut2}  Let $\sigma \in \Aut(\A)$ and
	assume  $a_{ij} =a\ne 0$ for all  $i<j$. If $\sigma(H)=H$, then $\sigma(e_1)=e_1$ and $\sigma(e_n)=e_n$.
\end{Le}
\begin{proof} Because, $\sigma(H)=H$, we have $\sigma(\Id_1(\A))=\Id_1(\A)$. Next, by Lemma \ref{le6} we know that  $e_1$ is the unique idempotent element of $\A$ with weight 1  satisfying $$e_1 e= \Big(\frac{1}{2}+a\Big)e_1+ \Big(\frac{1}{2}-a\Big)e$$
	for all $e\in \Id_1(\A)$ and $e_n$ is the unique idempotent element of $\A$ with weight 1 satisfying $$ee_n = \Big(\frac{1}{2}+a\Big)e+\Big(\frac{1}{2}-a\Big)e_n$$
	for all $e\in \Id_1(\A)$. This proves the lemma.
\end{proof}
\begin{Le}\label{laut3}
	 Assume  $a_{ij} =a $ for all  $i<j$. Then the linear mapping $\sigma_k$, for $2 \leq k \leq n-1$, defined by $\sigma_k(e_k)= e_{k-1}-e_k+e_{k+1}$ and  $\sigma_k(e_j) = e_j$ for all $j\ne k$   is an automorphism of the algebra $\A$. Furthermore, $\sigma_k\circ \sigma_k = \Id$.
\end{Le}
\begin{proof} Let $u= e_{k-1}-e_k+e_{k+1}$. Then $u\in \Id_1(\A)$, 
	$e_iu = e_ie_{k-1}-e_ie_k+e_ie_{k+1}= (\sfrac{1}{2}+a)e_i + (\sfrac{1}{2}-a)u$ and $ue_j =e_{k-1}e_j-e_ke_j+e_{k+1}e_j= (\sfrac{1}{2}+a)u + (\sfrac{1}{2}-a)e_j$ for all $i,j$ with  $1\leq i < k < j \leq n$. This shows that $\sigma_k$ is an automorphism of $\A$.
\end{proof}
Let $\A=\A(n;a)$ be a Lotka-Volterra algebra of dimension $n$, related to the  matrix $A=(a_{ij})$ with respect to the basis $\Phi=\{e_1,\ldots,e_n\}$ where $a_{ij}=a$ for all $i<j$.  It is clear that  $\A(n;a)\cong\A(n;-a)$ for all $a\in F$. If   $a\ne \pm \sfrac{1}{2}$, then  we have that  $\sigma(H)=H$ for all automorphisms $\sigma$ of $\A$.  Let $\sigma_k$ be the automorphism of $\A$ defined in previous lemma and denote by $\Ele(\A)$ the subgroup of $\Aut(\A)$ generated by all the mappings $\sigma_k$ of $\A$ defined in previous lemma. Each automorphism of $\A$ in   $\Ele(\A)$ can be represented in the following form
$$\sigma_{i_1i_2\cdots i_k}=\sigma_{i_1}\sigma_{i_2}\cdots \sigma_{i_k}$$
where $i_1,i_2, \ldots, i_k \in \{2,3,\ldots, n-1\}$.
An interesting question consists in to know when $\Aut(\A)=\Ele(\A)$.

We have a natural action of the symmetric group   $S_n$ to the set of all antisymmetric matrices of order $n$ given by $A_\tau = (a_{\tau(i),\tau(j)})$. If $\A$ is a Lotka-Volterra algebra with associated matrix $A$ with respect to the basis $\Phi=\{1,2,\ldots,n\}$, then  $\A$ is also a Lotka-Volterra algebra with associated matrix $A_\tau$    with respect to the basis $\Phi_\tau=\{e_{\tau(1)},e_{\tau(2)},\ldots, e_{\tau(n)} \}$.
 
\section{Dimension 3}
Let now $n=\dim(\A)=3$ and denote by $\Skew(a_{12},a_{13},a_{23})$ the antisymmetric matrix $A=(a_{ij})$. If $\Skew(a,b,c)$ is the associated matrix of a Lotka-Volterra algebra $\A$, then the new antisymmetric  matrices given in Table 1  are also associated matrices of  $\A$.
\begin{table}[h!]
	\centering
	\scriptsize \begin{tabular}{|c||c|c|c|c|c|c|}\hline
	\text{\scriptsize  Permutação }$\tau$ & $\Id$ & (12) & (13) & (23) & (123) & (132) \\ \hline
	\text{Matrix } $A_\tau$ &$\Skew(a,b,c)$&$\Skew(-a,c,b)$&$\Skew(-c,-b,-a)$&$\Skew(b,a,-c)$& $\Skew(-b,-c,a)$&$\Skew(c,-a,-b)$ \\ \hline
	\end{tabular}
 \caption{ Action of the symmetric group $S_3$}
	\label{table:1}
\end{table}

In the following, $\A(a,b,c)$  will denote a Lotka-Volterra algebra of dimension 3  associated to the antisymmetric matrix  $\Skew(a,b,c)$ with respect to the basis $\Phi=\{e_1,e_2,e_3\}$. We  previously showed that when $abc\ne 0$ and $b\ne a+c$, then $\A(a,b,c)$ has exactly seven nonzero idempotent elements
 $$\Big\{e_1,e_2,e_3,\frac{ce_1-be_2+a e_3}{a-b+c},\frac{1}{2a}\big(-e_1+e_2\big), \frac{1}{2b}\big(-e_1+e_3\big),
 \frac{1}{2c}\big(-e_2+e_3\big)\Big\}.$$
The following useful  lemma is an immediate consequence of Lemma \ref{l1} and Table 1.

\begin{Le}
	Let  $\dim\A=3$ and  assume that there exists $\sigma\in \Aut(\A)$ with $\sigma(H)\ne H$. Then $\A$ is isomorphic to one of the following  Lotka-Volterra algebras: $\A(-\sfrac{1}{2},-\sfrac{1}{2},c)$ or $\A(\sfrac{1}{2},\sfrac{1}{2},c)$  with $c\in F $.
\end{Le}
\begin{proof} Let $\chi= \omega\circ\sigma$. Because $\sigma(H)\ne H$, we have that $\chi$ is a weight homomorphism of $\A$ different from $\omega$. Let $I= \{e_i\in  \Phi \, | \, \chi(e_i) =0\}$. If $I$ contains exactly one element, then $\A$ is isomorphic to $\A(\sfrac{1}{2},\sfrac{1}{2},c)$ for some $c\in F$. In the other case, $I$ contains  2 elements and $\A$ is isomorphic to  $\A(-\sfrac{1}{2},-\sfrac{1}{2},c)$ for some $c\in F$. 
	\end{proof}
\begin{Le}
If  $ c \not\in \{0,\sfrac{1}{2},-\sfrac{1}{2}\}$, then 
$\Aut\big(\A(-\sfrac{1}{2},-\sfrac{1}{2},c)\big)= \{\Id,\eta\}$ where $\eta(e_1)=e_1$, $\eta(e_2)=e_1-e_3$ and $\eta(e_3)=e_1-e_2$.  
\end{Le}\begin{proof} Let $\A= \A(-\sfrac{1}{2},-\sfrac{1}{2},c)$ and take $\sigma \in \Aut(\A)$. 
Because $e_1$ is the unit element of the algebra $\A$,  we have that $\sigma(e_1)=e_1$.  Firstly, assume that  $\omega\circ\sigma = \omega$. Then $\sigma\big(\Id_1(\A)\big)=\Id_1(\A)$ and $\Id_1(\A)=\big\{e_1,e_2,e_3, e_4\big\}$ where $e_4=e_1 + \frac{1}{2c} \big(e_2- e_3\big)$. Because the support of $e_ke_4$ is $\{2,3\}$, we have that $e_ke_4$ is not a linear combination of $e_k$ and $e_4$ for $k=2,3$. Thus, 
 we can easily check that if  
 $u,v\in \Id_1(\A)\setminus \{e_1\}$, $u\ne v$ such that $uv=\big(\sfrac{1}{2}+c\big) u + \big(\sfrac{1}{2}-c\big)v$, then $u=e_2$ and $v=e_3$. This implies that $\sigma(e_2)=e_2$ and $\sigma(e_3)=e_3$ and hence $\sigma=\Id$.
 
We will now consider the other case, that is   $\chi=\omega\circ\sigma \ne \omega$. By Lemma 1, we have that  $\chi(\lambda_1 e_1 +\lambda_2 e_2 + \lambda_3 e_3)= \lambda_1$ for all $\lambda_i\in F$. Therefore, $\sigma(e_2),\sigma(e_3) \in \Id_0(\A) = \big\{e_1-e_2,e_1-e_3, \frac{1}{2c}\big(-e_2+e_3\big)\big\}$. Since the support of the product $(e_1-e_k)(-e_2+e_3)/(2c)$ is contained in the set $\{2,3\}$  for $k=2,3$ and  $\sigma(e_2)\sigma(e_3)=\big(\sfrac{1}{2}+c\big) \sigma(e_2) + \big(\sfrac{1}{2}-c\big)\sigma(e_3)$, we can  easily check that $\sigma(e_2)=e_1-e_3$ and $\sigma(e_3)=e_1-e_2$.  This completes the proof of the lemma.
\end{proof}
\begin{Le} We have that  $\A(-\sfrac{1}{2},-\sfrac{1}{2},-\sfrac{1}{2}) \cong \A(-\sfrac{1}{2},-\sfrac{1}{2},\sfrac{1}{2})$ and 
  $$\Aut\big(\A(-\sfrac{1}{2},-\sfrac{1}{2},-\sfrac{1}{2})\big)=\{\Id, \gamma, \gamma^2,\sigma_2, \sigma_2\gamma, \sigma_2\gamma^2\}$$
	where  $\gamma(e_1)=e_1$, $\gamma(e_2)=e_1-e_3$ and $\gamma(e_3)=e_2-e_3$. Furthermore, $\Aut\big(\A(-\sfrac{1}{2},-\sfrac{1}{2},-\sfrac{1}{2})\big)$ is isomorphic to  $S_3$.
\end{Le}
\begin{proof} Table 1 implies  that $\A(a,a,c)\cong \A(a,a,-c)$ for all $a,c\in F$. For simplicity, denote  $\A(-\sfrac{1}{2},-\sfrac{1}{2},-\sfrac{1}{2})$ by $\A$.
	Because  $\A$ is a unitary algebra with $e_1$ as unit element, we get that $\sigma(e_1)=e_1$ for all automorphisms $\sigma$  of $\A$. On the other hand, we observe that a linear homomorphism $\sigma$ of $\A$ where $\sigma(e_1)=e_1$ and $\sigma(e_j)\in \Id(\A)$ for $j=2,3$ is an automorphism  of $\A$ if and only  if $\sigma(e_2)\sigma(e_3) = \sigma(e_3)$. Thus, in order to determine $\Aut(\A)$ we  need to  find all the pairs $(u,v)$, $u\ne v$, of  idempotent elements in $\Id(\A)\setminus \{e_1\}$  such that $uv=v$. Using that $\Id(\A)=\{e_1,e_2,e_3,e_1-e_2+e_3,e_1-e_2,e_1-e_3,e_2-e_3\}$ we can easily check that such  pairs  are 
	$(e_2,e_3)$, $(e_2,u_{23})$, $(e_1-e_2+e_3, e_3)$,$(e_1-e_2+e_3, u_{12})$, $(u_{13},u_{12})$ and $(u_{13},u_{23})$ where $u_{ij}=e_i-e_j$.  Because $\sigma_2\circ \gamma \ne \gamma\circ \sigma_2$, we have that this automorphism group of order 6 is not abelian and  hence it is isomorphic to  $S_3$. This completes the proof of the lemma.
\end{proof}
\begin{Le}
	We have that $\Aut\big(\A(-\sfrac{1}{2},-\sfrac{1}{2},0)\big)=\{f_{a,b},g_{a,b} \,\, | \,\, a,b\in F, \,\, a\ne b\}$ where 
	$$f_{a,b}(e_1)=e_1, \quad 
	f_{a,b}(e_2)= e_1 -u_a,\quad 
	f_{a,b}(e_3)= e_1 -u_b,$$
	$$g_{a,b}(e_1)=e_1, \quad 
	g_{a,b}(e_2)=  u_a,\quad 
	f_{a,b}(e_3)= u_b$$
	and $u_a=ae_2+(1-a)e_3$.
\end{Le}
\begin{proof} Write $\A(-\sfrac{1}{2},-\sfrac{1}{2},0)$ by $\A$ for simplicity. Since $e_1$ is the unit element of the algebra $\A$, we have that $\sigma(e_1)=e_1$ for all $\sigma\in \Aut(\A)$. By Lemma \ref{l2}, the set $\Id(\A)$  is equal to  
	$\{e_1,u_a, e_1 - u_a \, |\, a\in F\}$ where  $u_a = ae_2 + (1-a)e_3$.
	On the other hand,  every element $u$ in $\A$ can be uniquely presented  by $u= \alpha e_1 + x$ where $\alpha\in F$ and $x$ is a linear combination of $e_2$ and $e_3$. If $v = \beta e_1 + y$ is another element of $ \A $, then the product $uv$ is given as follows:
	$$uv = (\alpha\beta)e_1 + \Big(\beta+ \frac{\omega(y)}{2}\Big) x+ 
	 \Big(\alpha+ \frac{\omega(x)}{2}\Big) y.$$
	 In particular, $u_a(e_1-u_b) = (\sfrac{1}{2})(u_a-u_b)$ so that $\det(u_a,e_1-u_b, u_a(e_1-u_b))\ne 0$ if $a\ne b$.
	 Now it is easy to check that two different idempotent elements $u,v\in \Id(\A)\setminus \{e_1\}$ satisfy $uv= (u+v)/2$ if and only if either $u=u_a$, $v=u_b$ or  $u= e_1 - u_a$ and $v= e_1 - u_b$ for some $a,b\in F$ with $a\ne b$. This proves the lemma.
\end{proof}

\begin{Le} We have that 
	\begin{eqnarray}
	&&	\label{rer} \A(\sfrac{1}{2},\sfrac{1}{2},-\sfrac{1}{2})\cong \A(\sfrac{1}{2},\sfrac{1}{2},\sfrac{1}{2}) \cong \A(-\sfrac{1}{2},-\sfrac{1}{2},\sfrac{1}{2}),\\
&&	\Aut(\A(\sfrac{1}{2},\sfrac{1}{2},c))= \{\Id\} \quad \text{if}\quad c \not\in \{0,\sfrac{1}{2},-\sfrac{1}{2}\}.\qquad 
	\end{eqnarray}
	
\end{Le}
\begin{proof} Relation  (\ref{rer}) is clear  from Table 1. Let  $c\ne 0, \pm\sfrac{1}{2}$ and denote $\A(\sfrac{1}{2},\sfrac{1}{2},c)$ by $\mathbf{A}_{c}$. From Lemma \ref{l2}, we know  that this algebra has exactly seven nonzero idempotent elements
	$$\Id(\mathbf{A}_c)= \Big\{e_1,e_2,e_3, e_4,-e_1+e_2,-e_1+e_3, \frac{1}{2c}(-e_2+e_3)\Big\}$$
	where $e_4=e_1+(\sfrac{1}{2c})(-e_2+e_3)$.
Notice that $e_1u= \omega(u) e_1$ for all $u \in \Id(\mathbf{A}_c)$, and  hence  $e_1$ is the unique nonzero idempotent $w$ of $\mathbf{A}_c$ such that $wu \in \{w, 0\}$ for all $u\in \Id(\mathbf{A}_c)$.  Therefore, $e_1$ is fixed for all automorphisms of $\mathbf{A}_c$. Let $\sigma\in\Aut(\A)$. We have that 
$e_1 = \sigma(e_1)= \sigma(e_1 e_i) = \sigma(e_1)\sigma(e_i)= e_1\sigma(e_i)=\omega(\sigma(e_i))e_1$ for $i=2,3$ so that 
$\sigma(e_2),\sigma(e_3) \in \Id_1(\A)\setminus \{e_1\}= \{e_2,e_3,e_4\}$. Now relation $\sigma(e_2)\sigma(e_3)=(\sfrac{1}{2}+c)\sigma(e_2)+ (\sfrac{1}{2}-c)\sigma(e_3)$ forces $\sigma(e_2)=e_2$ and $\sigma(e_3)=e_3$ since $\det(e_2,e_4,e_2e_4)\ne 0$ and $\det(e_3,e_4,e_3e_4)\ne 0$.
 This implies that $\sigma=\Id$.
\end{proof}

\begin{Le} We have that 
	$\Aut(\A(\sfrac{1}{2},\sfrac{1}{2},0))= \{f_{a,b},g_{a,b} \,\, | \,\, a,b\in F, \,\, a\ne b\}$ where 
	$$f_{a,b}(e_1)=e_1, \quad 
	f_{a,b}(e_2)= -e_1 + u_a,\quad 
	f_{a,b}(e_3)= -e_1 +u_b,$$
	$$g_{a,b}(e_1)=e_1, \quad 
	g_{a,b}(e_2)=  u_a,\quad 
	f_{a,b}(e_3)=  u_b.$$
and  $u_a = a e_2+ (1-a)e_3$.	
\end{Le} 
\begin{proof}	Let $\A=\A(\sfrac{1}{2},\sfrac{1}{2},0)$. By Lemma \ref{l2} we know that 
	$$\Id(\A)=\{e_1,u_a,-e_1+ u_a \,\, | \,\, a\in F\,\, \}.$$
	Every element $u\in\A$ can be uniquely presented  by $u= \alpha e_1 + x$ where $\alpha\in F$ and $x$ is a linear combination of $e_2$ and $e_3$. If $v=\beta e_1+y$ is another element of  $\A$,  then  
$$uv =(\alpha\beta +\beta\omega(x)+\alpha\omega(y))e_1 + \frac{\omega(y) x+\omega(x)y}{2}.$$
In particular, $e_1u_a=e_1$, $e_1(-e_1+u_a)=0$, $u_au_b=(u_a+u_b)/2$, $u_a(-e_1+u_b) = -e_1 + (u_a+ u_b)/2$ and $(-e_1+u_a)(-e_1+ u_b) = -e_1 + (u_a+u_b)/2$.
Therefore,   $e_1$ is the unique nonzero idempotent element $w$ of $\A$ such that $w\cdot\Id(\A)= \{0,w\}$ and hence $\sigma(e_1)=e_1$ for all $\sigma \in \Aut(\A)$. Furthermore, two different idempotents $u,v\in \Id(\A)\setminus \{e_1\}$ satisfy 
$uv=(u+v)/2$ if and only if either $u=u_a$ and $v=u_b$ or $u=-e_1+u_a$ and $v=-e_1+u_b$ for some $a,b\in \A$. This proves the lemma.
\end{proof}
\begin{teo}
	Let $\A$ be a Lotka-Volterra algebra with dimension 3 such that there exists   $\sigma \in \Aut(\A)$ where  $\sigma(H)\ne H$. Then $\A$ is isomorphic to one of the following  algebras $\A(-\sfrac{1}{2},-\sfrac{1}{2},c)$, $c\in F$, or  $\A(\sfrac{1}{2},\sfrac{1}{2},0)$. Furthermore, if $F$ is an infinite field and $\Aut(\A)$ is a finite group, then $\A\cong \A(-\sfrac{1}{2},-\sfrac{1}{2},c)$  for some particular $c\in F^*$.
\end{teo}

In the rest of this section we shall assume that $\A$ is a  Lotka-Volterra algebra with dimension 3   associated to the matrix $A=(a_{ij})$ with respect to the natural basis $\Phi=\{e_1,e_2,e_3\}$ such that $\A$ is not isomorphic to either $\A(-\sfrac{1}{2},-\sfrac{1}{2},c)$,  for any $c\in F$, or  $\A(\sfrac{1}{2},\sfrac{1}{2},0)$. Thus, every $\sigma \in \Aut(\A)$ satisfies $\sigma (H)=H$ and hence $\sigma(\Id_1(\A))= \Id_1(\A)$. Furthermore, we will suppose that $A\ne (0)$ and  $\Aut(\A) \ne \{\Id\}$. As usual in this paper, $u_a = ae_2 + (1-a)e_3$ and $u_{ij}=e_i-e_j$ for all $a\in F$ and $i,j\in\{1,2,3\}$.

\begin{Le}
 Let  $a_{13}= a_{12}+a_{23}$. Then  we have one of the following two possibilities:
 \begin{enumerate}[(i)]
 	\item Characteristic of $F$ is equal to 3 and $\A\cong \A(r,-r,r)$ for some particular  $r\in F^*$. In this case,  $\Aut(\A)= \{\Id, \varrho,\varrho^2\}$ where $\varrho(e_1)=e_{2}$, $\varrho(e_2)=e_{3}$  and $\varrho(e_3)=e_1$. 
 	\item  $\A\cong \A(r,r,0)$, for some particular  $r\ne 0, \pm 1/2$.
 	Furthermore, $$\Aut( \A(r,r,0))= \{g_{a,b}\,\, |\,\, a,b\in F,\,\, a\ne b\}$$ where $g_{a,b}(e_1)= e_1$,  $g_{a,b}(e_2)=u_a$ and   $g_{a,b}(e_3)=u_b$.
 \end{enumerate}
\end{Le}
\begin{proof}
 Renumbering the indices if necessary, we can assume that $a_{12}a_{13}\ne 0$. 
 Firstly, we will assume that  $a_{23}\ne 0$. By Lemma \ref{l2} we get  $\Id_1(\A)=\{e_1,e_2,e_3\}$. Take $\Id \ne \sigma \in \Aut(\A)$. Then $\sigma$ induces a permutation $\check{\sigma}\ne \Id$ over the set
	  $\{1,2,3\}$ where $\sigma(e_k) = e_{\check{\sigma}(k)}$ for $k=1,2,3$. Therefore, $A_{\check{\sigma}}=A$. Now, Table 1 implies that $\check{\sigma}\in \{(123), (132)\}$ and  $a_{12}=a_{23}=-a_{13}$. Because $a_{13}=a_{12}+a_{23}$, we have $3a_{12}=0$ and hence $A$ has characteristic 3.
	  
	  Let now $a_{23}=0$. This forces  $a_{12}=a_{13}$ and by hypothesis $a_{12} \ne 0, \pm\sfrac{1}{2}$. We already know that $\Id_1(\A)= \{e_1, u_a \,\, | \,\, a\in F\}$. Besides,  $e_1$ is the unique idempotent $w$ in $\Id_1(\A)$ satisfying $wu= (\sfrac{1}{2}+a_{12})w + (\sfrac{1}{2}-a_{12})u$ for all $u\in \Id_1(\A)$. This forces $\sigma(e_1)=e_1$. Finally, we observe that $u_au_b = (u_a+u_b)/2$ for all $a,b\in F$. This completes the proof of the lemma.
\end{proof}
Now we will consider the case  $a_{13}\ne a_{12}+a_{23}$. By Table 1, we have that  $\A(a_{12},a_{13},a_{23})$ is isomorphic to  one of the following algebras: (i) $\A(a,b,c)$ where $abc\ne 0$ and $b\ne a+c$; (ii)   $\A(a,b,0)$ where $b\ne a$; (iii) $\A(a,0,0)$ where $a\ne 0$. 
 Every case will be analyzed separately.
\begin{Le}
	 If $a_{13}\ne a_{12}+a_{23}$ and $a_{13} a_{12}a_{23}\ne 0$, then $\A$ is isomorphic to one of the following algebras:
	 \begin{enumerate}[(i)]
	 	\item $\A(r,-r,r)$ with $r\ne 0$ and characteristic of $F$ different from 3; 
	 	\item $\A(r,r,r)$ with $r\ne 0,\pm \sfrac{1}{2}$.
	 \end{enumerate}
 Furthermore, $\Aut(\A(r,-r,r))=\{\Id,\varrho,\varrho^2\}$ and  $\Aut(\A(r,r,r))=\{\Id,\sigma_2\}$ where $\varrho(e_1)= e_{2}$, $\varrho(e_2)=e_3$ and $\varrho(e_3)=e_1$.
\end{Le}
\begin{proof} By Lemma \ref{l2} we have that $\Id_1(\A)=\{e_1,e_2,e_3,e_4\}$ where  $$e_4=\frac{1}{a_{12}-a_{13}+a_{23}} \big( a_{23}e_1 -a_{13}e_2 +a_{12}e_3\big).$$ 
By  hypothesis $\sigma(\Id_1(\A)) = \Id_1(\A)$.  If $e_4$ is fixed by all  $\sigma \in \Aut(\A)$, then    Table 1 forces $a_{12}=a_{23}=-a_{13}$ and 
$\Aut(\A)= \{\Id, \varrho, \varrho^2\}$.
Assume now that there exists $\sigma \in\Aut(\A)$ such that $\sigma(e_4)\ne e_4$. Renumbering the indices if it is necessary, we can assume that
$\sigma(\{e_1,e_2,e_3\})=\{e_1,e_3,e_4\}$. This implies that 
$e_1e_4$ is a linear combination of $e_1$ and $e_4$. Therefore,
$$0=\text{det}_{\Phi}(e_1,e_4,e_1e_4)= -\frac{a_{12}a_{13}(a_{12}-a_{13})}{(a_{12}-a_{13}+a_{23})^2}$$
so that 
 $a_{12}=a_{13}$. Analogously,  we have that  $e_3e_4$ is a linear combination of $e_3$ and $e_4$ and this implies $a_{23}=a_{13}$. Thus, we have that 
$$a_{12}=a_{13}=a_{23}.$$
Furthermore, because $\sigma(H)=H$ for all  $\sigma\in \Aut(\A)$ we get $a_{12}\ne \pm \sfrac{1}{2}$. Using Lemma \ref{laut2} and Lemma \ref{laut3}, we immediately get that  $\Aut(\A)=\{\Id,\sigma_2\}$. This completes the proof  of the lemma.
\end{proof}

\begin{Le}
	If $a_{12}\ne a_{13}$, $a_{23}=0$, then  $a_{12} a_{13}= 0$.
\end{Le}
\begin{proof} Suppose for the sake of contradiction that it is not true that $a_{12} a_{13}= 0$. Then $a_{12} a_{13}\ne 0$ and $\Id_1(\A) = \{ e_1, u_a :  a\in F\}$ where $u_a= a e_2 + (1-a)e_3$. Since 
	$$\text{det}_\Phi(e_1,u_a,e_1u_a)= a(1-a)(a_{12}-a_{13})$$
we can affirm that $\sigma\big(\{e_1,e_2,e_3\}\big)=\{e_1,e_2,e_3\}$ for all $\sigma\in \Aut(\A)$ and hence  $\Aut(\A)=\{\Id\}$ in contradiction with the hypothesis. The contradiction means that it is impossible $a_{12} a_{13}\ne 0$.
\end{proof}
\begin{Le}
	If $a_{12}\ne 0$ and  $ a_{13}=a_{23}=0$, then 
	$\Aut(\A)=\{\Id,\rho\}$ where $\rho(e_j)=-e_j+ 2 e_3$ for $j=1,2,3$.
\end{Le}
\begin{proof}  Let $\sigma\in \Aut(\A)$. We know that $\Id_1(\A) = \{  e_3, v_a, u_a :  a\in F^*\}$ where $v_a= a e_1 + (1-a)e_3$ and $u_a= a e_2 + (1-a)e_3$. Since  $e_3$ is the unique idempotent $w$ in $\Id_1(\A)$ satisfying $wu=(w+u)/2$ for all $u \in \Id_1(\A)$ we get   $\sigma(e_3) = e_3$. Thus, $\sigma(e_1),\sigma(e_2)\in \{v_a, u_a :  a\in F^*\}$ and $\sigma(e_1)\sigma(e_2) = (\sfrac{1}{2}+a_{12}) \sigma(e_1) + (\sfrac{1}{2}-a_{12}) \sigma(e_2)$. Since $v_av_b=(v_a+v_b)/2$ and  $u_au_b=(u_a+ u_b)/2$  for all $a,b\in F^*$, we obtain that  $\{\sigma(e_1),\sigma(e_2)\} = \{v_a,u_b\}$ for some $a,b\in F^*$. Next, 
\begin{eqnarray*}
&&0=\text{det}_\Phi(v_a,u_b,v_au_b)= a_{12}ab(a-b) \implies a=b,\\
&&v_au_a = \Big(\frac{1}{2}+a_{12} a\Big)v_a+\Big( \frac{1}{2}-a_{12} a\Big)u_a  \implies a= \pm 1.
\end{eqnarray*}  
If $a=1$, then $\sigma = \Id$ and for $a=-1$ we get $\sigma = \rho$. 
 This completes the proof of the lemma.
\end{proof}
	
Summarizing the previous results  we get the following classification.
\begin{teo} Let $\A$ be a Lotka-Volterra algebra over a field $F$ with characteristic different from 2 and  dimension 3. If  $\Aut(\A) \ne \{\Id\}$, then $\A$ is isomorphic to one of the following algebras:  $\A(0,0,0)$, $\A(r,0,0)$, $\A(r,r,0)$, $\A(-\sfrac{1}{2},-\sfrac{1}{2},c)$, $\A(r,r,r)$ and $\A(r,-r,r)$ where $c\in F$ and $r\in F^*$. Furthermore,
\begin{enumerate}[(i)]
	\item $\Aut\big(\A(0,0,0)\big)$ is the set of all linear automorphisms $\sigma$  of $\A$   such that $\sigma(H)=H$.
	\item If $r\ne 0$, then $\Aut\big(\A(r,0,0)\big) =\{\Id, \rho\}$ where $\rho(e_k)=-e_k+ 2e_3$ for $k=1,2,3$.
	\item If $r\ne 0, \pm\sfrac{1}{2}$, then $\Aut((\A(r,r,0) )= \{g_{a,b}\, | a,b\in F, \, a\ne b\},$ where $g_{a,b}(e_1)= e_1$, $g_{a,b}(e_2)= u_a$,
	$g_{a,b}(e_3)= u_b$. 
	\item If $\epsilon=\pm 1$, then 
	$\Aut\big(\A(\sfrac{\epsilon}{2},\sfrac{\epsilon}{2},0 )\big) = \big\{f^{(\epsilon)}_{a,b}, g_{a,b}\, | \, a,b\in F, \, a\ne b\big\}$ where  $f^{(\epsilon)}_{a,b}(e_1)=e_1$ and $f^{(\epsilon)}_{a,b}(e_k)= \epsilon(-e_1+ g_{a,b}(e_k))$ for $k=2,3$.
	\item $\Aut\big(\A(r,r,r)\big) =\{\Id, \sigma_2\}$ for  $r\ne 0,\pm \sfrac{1}{2}$.
	\item $\Aut\big(\A(r,-r,r)\big) =\{\Id, \varrho,\varrho^2\}$ where $\varrho(e_1)=e_{2}$, $\varrho(e_2)=e_{3}$  and $\varrho(e_3)=e_1$ for $r\ne 0$.
	\item If $c\ne 0,\pm\sfrac{1}{2}$, then $\Aut\big(\A(-\sfrac{1}{2},-\sfrac{1}{2},c)\big)=\{\Id,\eta \}$ where $\eta(e_1)=e_1$, $\eta(e_2)= e_1-e_3$ and $\eta(e_3)=e_1-e_2$.
	\item $\Aut\big(\A(-\sfrac{1}{2},-\sfrac{1}{2},-\sfrac{1}{2})\big)=
    \{\Id, \gamma, \gamma^2,\sigma_2, \sigma_2\gamma, \sigma_2\gamma^2\}\cong S_3$
	where  $\gamma(e_1)=e_1$, $\gamma(e_2)=e_1-e_3$ and $\gamma(e_3)=e_2-e_3$. 	
\end{enumerate}
\end{teo}

 \section{Acknowledgements}
The first author was supported by FAPESP, Proc. 2014/09310-5.


\end{document}